\documentclass[12pt]{article}
\usepackage{amsmath}
\usepackage{amssymb}
\usepackage{amsthm}
\usepackage{mathrsfs}
\usepackage{mathtools}
\usepackage[top=1in, bottom=1in, left=0.8in, right=1in]{geometry}
\usepackage{multicol}
\usepackage{wrapfig}
\usepackage[retainorgcmds]{IEEEtrantools}
\usepackage{upgreek}
\usepackage[T1]{fontenc}
\usepackage{hyperref}
\usepackage{stmaryrd}
\usepackage{xcolor}
\usepackage[backend=biber]{biblatex}
\usepackage{tikz-cd}
\usepackage{xy}
\input xy
\xyoption{all}

\newcommand{\Z}{\mathbb Z}

\newcommand{\N}{\mathbb N}

\newcommand{\F}{\mathbb F}

\DeclareMathOperator{\im}{im}

\DeclareMathOperator{\Spec}{Spec}

\DeclareMathOperator{\Sat}{Sat}
\DeclareMathOperator{\Ann}{Ann}

\newcommand{\W}{\mathcal{W}}

\newcommand{\cris}{\mathrm{cris}}

\newcommand{\trdeg}{\mathrm{tr deg}}

\makeatletter
\newcommand{\colim@}[2]{%
  \vtop{\m@th\ialign{##\cr
    \hfil$#1\operator@font colim$\hfil\cr
    \noalign{\nointerlineskip\kern1.5\ex@}#2\cr
    \noalign{\nointerlineskip\kern-\ex@}\cr}}
}
\newcommand{\colim}{%
  \mathop{\mathpalette\colim@{\rightarrowfill@\scriptscriptstyle}}\nmlimits@
}
\newcommand{\lim@}[2]{%
  \vtop{\m@th\ialign{##\cr
    \hfil$#1\operator@font lim$\hfil\cr
    \noalign{\nointerlineskip\kern1.5\ex@}#2\cr
    \noalign{\nointerlineskip\kern-\ex@}\cr}}
}
\newcommand{\limarrow}{%
  \mathop{\mathpalette\lim@{\leftarrowfill@\scriptscriptstyle}}\nmlimits@
}
\renewcommand{\varprojlim}{%
  \mathop{\mathpalette\varlim@{\leftarrowfill@\scriptscriptstyle}}\nmlimits@
}
\renewcommand{\varinjlim}{%
  \mathop{\mathpalette\varlim@{\rightarrowfill@\scriptscriptstyle}}\nmlimits@
}
\makeatother

\newcounter{ctr}[section]
\numberwithin{ctr}{section}
\numberwithin{equation}{ctr}

\newcounter{ltr}

\theoremstyle{plain}

\newtheorem{theorem}[ctr]{Theorem}
\newtheorem{proposition}[ctr]{Proposition}

\newtheorem{lemma}[ctr]{Lemma}
\newtheorem{corollary}[ctr]{Corollary}

\newtheorem{intro_theorem}[ltr]{Theorem}
\newtheorem{intro_corollary}[ltr]{Corollary}

\theoremstyle{definition}

\newtheorem{example}[ctr]{Example}
\newtheorem{remark}[ctr]{Remark}
\newtheorem{para}[ctr]{}
\newtheorem{question}[ctr]{Question}

\title{Dimensional vanishing of the saturated de Rham--Witt complex}
\author{Ravi Fernando -- \texttt{ravif@illinois.edu}}
\date{}
\begin{document}
\maketitle

\begin{abstract}
The saturated de Rham--Witt complex, introduced by Bhatt--Lurie--Mathew in \cite{BLM}, is a variant of the classical de Rham--Witt complex which is expected to behave better for singular schemes.
We provide partial justification for this expectation by showing that the saturated de Rham--Witt complex satisfies a dimensional vanishing property even in the presence of singularities.  This is stronger than the vanishing properties of the classical de Rham--Witt complex, crystalline cohomology, or de Rham cohomology, and is instead comparable to \'etale cohomology.
\end{abstract}

\section{Introduction}

\begin{para}
Let $k$ a perfect field of characteristic $p \neq 0$, $W = W(k)$ its ring of Witt vectors, and $X$ a scheme over $k$.  Then we can associate to $X/k$ two closely related commutative differential graded algebras over the Witt vector sheaf $W \mathcal O_X$:  the \emph{classical de Rham--Witt complex}
$$
W \Omega^*_X = \lim_{\leftarrow r} W_r \Omega^*_X
$$
of \cite{illusie}, and the \emph{saturated de Rham--Witt complex}
$$
\W\Omega^*_X = \lim_{\leftarrow r} \W_r \Omega^*_X
$$
of \cite{BLM}.  If $X/k$ is smooth, then we have a canonical isomorphism $W \Omega^*_X \overset{\sim}{\to} \W \Omega^*_X$, and both de Rham--Witt complexes are representatives of the derived Zariski sheaf $Ru_*(\mathcal O_{X/W, \cris})$ which computes the crystalline cohomology of $X/W$.  Moreover, each of these statements comes from an analogous comparison between the $p$-adic pro-objects $(W_r \Omega^*_X)_{r>0}$, $(\W_r \Omega^*_X)_{r>0}$, and $(Ru_*(\mathcal O_{X/W_r(k), \cris}))_{r>0}$.
\end{para}

\begin{para}
However, without the assumption of smoothness, the objects $Ru_*(\mathcal O_{X/W, \cris})$, $W \Omega^*_X$, and $\W\Omega^*_X$ can in general be pairwise distinct, even up to quasi-isomorphism.
Since the appearance of \cite{BLM}, there has generally been an expectation that the saturated de Rham--Witt complex should produce ``better'' results than the classical one for singular schemes $X/k$.
\end{para}

\begin{para}
However, the precise results in this direction have been somewhat limited.
Accordingly, let us briefly recall some of the pleasant properties of $\W\Omega^*_X$ which are not necessarily shared by $W\Omega^*_X$ when $X/k$ is singular.
(We refer the reader to \cite{illusie_talks} for a more comprehensive summary of the literature.)
First, the saturated de Rham--Witt complex is definitionally always $p$-torsionfree, and \cite{BLM} proves that it is invariant under passing to the reduction of $X$, or even the seminormalization.  None of these properties hold for the classical de Rham--Witt complex, as they are all inconsistent with the identity $W\Omega^0_X = W\mathcal O_X$.
Second, Ogus (\cite{ogus}) has undertaken an extensive study of $\W\Omega^*_X$ for varieties $X/k$ with toroidal singularities, proving among other things that for any reduced, smooth idealized fs log scheme $X/k$, $\W_1 \Omega^*_X$ recovers a previously studied subcomplex of the log de Rham complex of $X$.
\end{para}

\begin{para}
The goal of this paper is to document a general sense in which the saturated de Rham--Witt complex behaves ``better'' than either the classical de Rham--Witt complex or crystalline cohomology:  namely, it satisfies dimensional vanishing even for singular varieties.
Previously, the strongest known result in this direction was \cite[Lemma 9.3.7(3)]{BLM}:
\end{para}

\begin{lemma} \label{lem:BLM_dimensional_vanishing}
(Bhatt--Lurie--Mathew)
Suppose that $R/k$ is a commutative algebra such that $\Omega^1_{R/k}$ is locally generated by $\leq d$ elements.  Then $\W\Omega^*_R$ is concentrated in degrees $[0,d]$.
\end{lemma}

\begin{remark}
If $R/k$ is of finite type, then the optimal choice of $d$ above is the local embedding dimension of $R/k$; i.e. the largest dimension among its tangent spaces (cf. \cite[tag \texttt{0C2G}]{stacks}).
\end{remark}

\begin{remark}
Again assuming $R/k$ is of finite type, it is not difficult to prove the analogue of Lemma \ref{lem:BLM_dimensional_vanishing} for the classical de Rham--Witt complex.  Indeed, if $\Omega^1_{R/k}$ is locally generated by $\leq d$ elements, then Nakayama's lemma\footnote{Our use of Nakayama's lemma requires us to have prior knowledge that $\Omega^1_{W_n(R)/W_n(k)}$ is a finitely generated $W_n(R)$-module.  It is not clear whether this is a formal consequence of the finite generation of $\Omega^1_{R/k}$.  But when $R/k$ is of finite type, the same is true of $W_n(R)/W_n(k)$, so $\Omega^1_{W_n(R)/W_n(k)}$ is finitely generated.} implies that these lift to local generators of $\Omega^1_{W_n(R)/W_n(k)}$ for any $n$.  Thus $\Omega^*_{W_n(R)/W_n(k)}$ is concentrated in degrees $[0,d]$, which implies that its quotient $W_n \Omega^*_R$ is as well, and thus also the limit $W\Omega^*_R$.
\end{remark}

\begin{para}
We will show the following surprising improvement of Lemma \ref{lem:BLM_dimensional_vanishing}, which does \emph{not} hold for the classical de Rham--Witt complex.  We note that while Bhatt--Lurie--Mathew's proof uses the machinery of the derived de Rham--Witt complex, ours is quite elementary:  we first prove by hands-on calculations that $\W_1 \Omega^n_{k[x_1, \dots, x_n]/I}$ vanishes for any $n \geq 0$ and any ideal $I \neq 0$ (which is enough to prove our claim in the case of a hypersurface), and we then reduce to this case by a transcendence degree argument.
\end{para}

\begin{intro_theorem} \label{thm:dimensional_vanishing_intro}
(See Corollary \ref{cor:dimensional_vanishing_finite-type}.)
If $R$ is a finite-type algebra over a perfect field $k$, then $\W\Omega^*_R$ is concentrated in degrees $[0, \dim R]$.
\end{intro_theorem}

\begin{remark}
This bound is sharp, since (after replacing $R$ by its reduction) we have $\W\Omega^* = W\Omega^*$ over the smooth locus.  Note that $\dim R$ is always less than or equal to the local embedding dimension of $R/k$, and it is strictly less when $R$ is equidimensional and non-smooth.
\end{remark}

Theorem \ref{thm:dimensional_vanishing_intro} is a consequence of the following more general statement:

\begin{intro_theorem} \label{thm:noetherian_dimensional_vanishing_intro}
(See Theorem \ref{thm:noetherian_dimensional_vanishing}.)
Suppose $R$ is a Noetherian $k$-algebra.  The complex $\W\Omega^*_R$ is concentrated in degrees $[0, \sup_{\eta}(\trdeg(\kappa(\eta)/k))]$, where $\eta$ runs over the generic points of $\Spec R$.
\end{intro_theorem}

Globalizing and passing to hypercohomology, Theorem \ref{thm:dimensional_vanishing_intro} yields the corollary:

\begin{intro_corollary} \label{cor:dimensional_vanishing_coh_intro}
(See Corollary \ref{cor:dimensional_vanishing_coh}.)
For any finite-type scheme $X/k$, the hypercohomology $\mathbb H^*(X, \W\Omega^*_X)$ is concentrated in degrees $[0, 2 \dim X]$.  If $X/k$ is moreover affine, then $\mathbb H^*(X, \W\Omega^*_X)$ is concentrated in degrees $[0, \dim X]$.  The same bounds hold for the hypercohomology of the quotient complex $\W_r \Omega^*_X$ for any $r > 0$.
\end{intro_corollary}

\begin{remark}
Corollary \ref{cor:dimensional_vanishing_coh_intro} is directly analogous to the behavior of \'etale cohomology for (possibly singular) varieties over a separably closed field; see \cite[Theorems VI.1.1 and VI.7.2]{milne_book}.
Note that it fails for both the classical de Rham--Witt complex and the crystalline cohomology of a cusp, by \cite[\textsection \textsection 6.3-6.4]{BLM}.
In fact, the analogue of Theorem \ref{thm:dimensional_vanishing_intro} for the de Rham complex is \emph{always} false for equidimensional singular varieties, regardless of the base field.   (It follows that the same is true of the classical de Rham--Witt complex, of which de Rham is a quotient.)  Indeed, if $x \in X$ is a point with tangent dimension exceeding $n = \dim X$ and we choose functions $f_1, \dots, f_{n+1} \in \mathfrak m_x$ which are linearly independent modulo $\mathfrak m_x^2$, then the section $df_1 \wedge \cdots \wedge df_{n+1} \in \Omega^{n+1}_X$ does not vanish at $x$.
\end{remark}

\begin{remark}
The construction $R \mapsto \W_r \Omega^*_R$ commutes with Zariski localizations (\cite[Remark 5.2.3]{BLM}) and filtered colimits of rings (\cite[Corollary 4.3.5]{BLM}), thus also with stalk-localization.
In particular, since $\W_r \Omega^n_R$ vanishes if and only if all of its stalks vanish, Theorem \ref{thm:dimensional_vanishing_intro} is equivalent to the analogous statement for local $k$-algebras which are essentially of finite type (i.e. local rings of finite-type $k$-algebras).  But in fact Theorem \ref{thm:noetherian_dimensional_vanishing_intro} applies directly to both finite-type $k$-algebras and their localizations, so there is no need to pass through the world of local rings.
\end{remark}

\begin{example}
Recall from \cite[Theorem 4.3.1]{BLM} that for every regular Noetherian $k$-algebra $R$, the natural map $\Omega^*_{R/k} \to \W_1 \Omega^*_R$ is an isomorphism.
In particular, if $R$ is the imperfect field $k(t_1, \dots, t_n)$ for some $n > 0$, it follows that
$$\W_1 \Omega^i_R = \Omega^i_{R/k} = \Omega^i_{k[t_1, \dots, t_n]} \otimes_{k[t_1, \dots, t_n]} k(t_1, \dots, t_n),$$
which is nonzero for each $i = 0, \dots, n$.
This shows that Theorem \ref{thm:dimensional_vanishing_intro} does not hold without the finite-type hypothesis.  In particular, the dimensional vanishing of $\W\Omega^*_R$ is not an intrinsic property of Krull dimension, but rather a statement about transcendence degree over a perfect base field.
\end{example}

\begin{para}
\textbf{Acknowledgements.}  I am grateful to Chris Dodd for helpful discussions about this work.
\end{para}

\section{Preliminaries on Dieudonn\'e complexes}

We begin with a few elementary observations about saturated Dieudonn\'e complexes.

\begin{lemma} \label{lem:F_inj_ill-defined}
Suppose $M^*$ is a Dieudonn\'e complex, and fix integers $r > 0$ and $n$.  If $x \in M^n$ is such that $Fx$ maps to $0$ in $\W_r \Sat(M^*)$, then $x$ does as well.
\end{lemma}
\begin{proof}
First note that we may assume without loss of generality that $M^*$ is saturated, as the canonical map $M^* \to \Sat(M^*)$ is compatible with $F$.  Now suppose $M^*$ is saturated and we are given such an $x \in M^n$.  Then we have $Fx = V^r y + dV^r z$ for some $y$ and $z$.  Applying $F^r d$ to both sides yields $F^r dFx = F^r dV^r y = dy$.  The left side is divisible by $p$ thanks to the identity $dF = pFd$.  By the definition of saturatedness, this implies that $y$ is in the image of $F$; say $y = Fw$.  Then we have
$$
Fx = V^r F w + dV^r z = F V^r w + F d V^{r+1} z,
$$
and since $F$ is injective, it follows that $x = V^r w + dV^r(Vz)$.
\end{proof}

\begin{remark}
We caution the reader that the converse of Lemma \ref{lem:F_inj_ill-defined} is not generally true, because (taking $M^*$ saturated) $F(dV^r x) = dV^{r-1} x$ does not generally lie in $\im V^r  + \im dV^r$.  In other words, $F$ does not generally induce a well-defined map $\W_r M^n \to \W_r M^n$.  If it did, the lemma would say that this map is injective.  Indeed, \cite[Lemma 3.6.1]{BLM} can be interpreted in this light:  when $A^*$ is a saturated Dieudonn\'e algebra, $F$ \emph{does} induce a well-defined map $\W_1 A^0 \to \W_1 A^0$ (because $A^{-1} = 0$); this map is just the Frobenius endomorphism of the $\F_p$-algebra $\W_1 A^0$, and it is injective by the lemma above.
\end{remark}

\begin{proposition} \label{prop:W1_suffices}
Suppose $M$ is a saturated Dieudonn\'e complex, and fix $n \in \Z$.  If $\W_1 M^n = 0$, then $\W M^n = 0$.
\end{proposition}
\begin{proof}
Recall from \cite[Proposition 2.7.1]{BLM} that for any $r$ and $n$, we have an isomorphism $\W_r M^n \simeq H^n(M^*/p^r M^*)$ induced by $F^r$.
By considering the long exact sequence associated to the short exact sequence of complexes
$$
0 \to M^*/pM^* \overset{p^r}{\to} M^*/p^{r+1} M^* \to M^*/p^r M^* \to 0,
$$
we see that the vanishing of $H^n(M^*/pM^*)$ and $H^n(M^*/p^r M^*)$ implies that of $H^n(M^*/p^{r+1} M^*)$.  Thus, if $\W_1 M^n \simeq H^n(M^*/pM^*) = 0$, it follows by induction that $\W_r M^n \simeq H^n(M^*/p^r M^*) = 0$ for all $r$, and thus $\W M^n = 0$.
\end{proof}

\section{The case of hypersurfaces}

\begin{para}
Throughout this section, we fix a finite-type $k$-algebra $R$ equipped with a presentation $R = k[x_1, \dots, x_n]/I$.  We would like to show that if the ideal $I$ is nonzero, then the degree-$n$ component of $\W\Omega^*_R$ vanishes.  For convenience, we will assume (without loss of generality) that $R$ is reduced.
\end{para}

\begin{para} \label{para:dRW_and_dR_W_setup}
Recall from \cite[Proposition 4.1.4]{BLM} that $\W\Omega^*_R$ may be constructed as the strictification of the (naive) de Rham complex $\Omega^*_{W(R)}$.  This gives rise to a commutative diagram
\[
\begin{tikzcd}
{\Omega^*_{W(k[x_1, \dots, x_n])}} \arrow[r, two heads] \arrow[d, two heads] & \Omega^*_{W(R)} \arrow[r, "\rho"] \arrow[d, two heads] & \W\Omega^*_R \arrow[d, two heads] \\
{\Omega^*_{k[x_1, \dots, x_n]}} \arrow[r, two heads]                         & \Omega^*_R \arrow[r]                                   & \W_1\Omega^*_R                   
\end{tikzcd}
\]
of commutative differential graded algebras, where the top row is a diagram of Dieudonn\'e algebras and $\rho$ is the strictification map.  Note that $\Omega^n_{k[x_1, \dots, x_n]}$ is a free $k[x_1, \dots, x_n]$-module of rank $1$ with generator $dx_1 \wedge \cdots \wedge dx_n$.  Our main goal is to prove the following proposition:
\end{para}

\begin{proposition} \label{prop:omega_vanishes}
If $R = k[x_1, \dots, x_n]/I$ as above with $I \neq 0$, then the element
$$\omega := dx_1 \wedge \cdots \wedge dx_n \in \W_1 \Omega^n_R$$
vanishes.
\end{proposition}

\begin{remark}
A priori, Proposition \ref{prop:omega_vanishes} is slightly weaker than the vanishing of $\W_1 \Omega^n_R$ (or $\W\Omega^n_R$), since the canonical map $\Omega^*_R \to \W_1 \Omega^*_R$ is not generally surjective.
It turns out that this is not a serious obstacle, thanks to the following lemma.
Thus, although it is Proposition \ref{prop:omega_vanishes} that will actually be used in the next section, the reader may rest assured that this proposition does in fact imply a vanishing result for $\W\Omega^*_R$.
\end{remark}

\begin{lemma} \label{lem:vanishing_map_suffices}
Suppose $R$ is any $\F_p$-algebra and $n$ is an integer such that the natural map $\Omega^n_R \to \W_1 \Omega^n_R$ vanishes identically.  Then $\W\Omega^n_R = 0$.
\end{lemma}
\begin{proof}
Assume without loss of generality that $R$ is reduced.  We will prove that $\W_1 \Omega^n_R = 0$; this implies the claim by Proposition \ref{prop:W1_suffices}.  Note that $\W_1 \Omega^*_R = \W_1 \Sat(\Omega^*_{W(R)})$, so it is enough to show that every element of $\Sat(\Omega^*_{W(R)})^n$ vanishes upon passing to $\W_1$.

Let $x \in \Sat(\Omega^*_{W(R)})^n$ be arbitrary.  By \cite[Remark 2.3.3]{BLM}, $F^N x$ lies in the image of $\Omega^*_{W(R)}$ for some sufficiently large $N$.  Since the map $\Omega^*_{W(R)} \to \W_1 \Omega^*_R$ factors through $\Omega^*_R$, it follows from our assumption that $F^N x$ maps to $0$ in $\W_1 \Omega^*_R$.  So Lemma \ref{lem:F_inj_ill-defined} implies that $x$ vanishes in $\W_1 \Omega^*_R$ as well, completing the proof.
\end{proof}

\begin{para} \label{para:annihilator}
To prove Proposition \ref{prop:omega_vanishes}, we will view $\W_1 \Omega^n_R$ as a $k[x_1, \dots, x_n]$-module via the map $\Omega^*_{k[x_1, \dots, x_n]} \to \W_1 \Omega^*_R$, and we will study the ideal of $k[x_1, \dots, x_n]$ which annihilates $\omega$.
Clearly, this annihilator contains $I$, as any $f \in I$ annihilates $\omega$ already in $\Omega^*_R$.
To enlarge the annihilator beyond this, we will use Lemma \ref{lem:F_inj_ill-defined}.  This will require temporarily working with elements of $\W\Omega^*_R$ rather than its quotient $\W_1 \Omega^*_R$.
\end{para}

\begin{lemma} \label{lem:F_mult_lift}
Let $g \in R$, and consider the elements $[g], d[g] \in \Omega^*_{W(R)}$, where $[ \cdot ]$ denotes the multiplicative lift.  We have $F([g]) = [g^p] = [g]^p$ and $F(d[g]) = [g]^{p-1} d[g]$.
\end{lemma}
\begin{proof}
The first claim is a standard fact about multiplicative lifts in $W(R)$, and the second follows from it using the formula of (\cite[Proposition 3.2.1]{BLM}).
\end{proof}

\begin{proposition} \label{prop:radical_and_partials}
The annihilator of $\omega \in \W_1 \Omega^n_R$ is a radical ideal of $k[x_1, \dots, x_n]$ which is closed under taking partial derivatives.
\end{proposition}
\begin{proof}
To prove $\Ann(\omega)$ is a radical ideal, suppose $g^p \omega = 0$ for some polynomial $g$.  By multiplying, we see that $g^p (x_1 \cdots x_n)^{p-1} \omega$ also vanishes.  In particular, the element
$$
[g]^p ([x_1] \cdots [x_n])^{p-1} d[x_1] \wedge \cdots \wedge d[x_n] \in \Omega^n_{W(R)}
$$
maps to $0$ in $\W_1 \Omega^n_R$.  But this element equals
$$
F([g] d[x_1] \wedge \cdots \wedge d[x_n])
$$
by Lemma \ref{lem:F_mult_lift} and the multiplicativity of $F$.  Thus, by applying Lemma \ref{lem:F_inj_ill-defined} and passing back down to $\W_1 \Omega^n_R$, we can conclude that $g \omega$ vanishes.

As for closure under partial derivatives, suppose $g \omega = 0$ for some polynomial $g$, and let $g_{x_i}$ be its partial derivative for some index $1 \leq i \leq n$.  Then the element
$$
\left( g \prod_{j \neq i} x_j \right)^{p-1} g_{x_i} \cdot \omega = \left( g \prod_{j \neq i} x_j \right)^{p-1} dx_1 \wedge \cdots \wedge dg \wedge \cdots \wedge dx_n
$$
also vanishes in $\W_1 \Omega^n_R$, where we have inserted $dg$ in place of $dx_i$.  As before, we can lift this to the element
\begin{align*}
& \left( [g] \prod_{j \neq i} [x_j] \right)^{p-1} d[x_1] \wedge \cdots \wedge d[g] \wedge \cdots \wedge d[x_n] \\
& = F \left( d[x_1] \wedge \cdots \wedge d[g] \wedge \cdots \wedge d[x_n] \right) \in \Omega^n_{W(R)},
\end{align*}
so Lemma \ref{lem:F_inj_ill-defined} implies that the element
$$
dx_1 \wedge \cdots \wedge dg \wedge \cdots \wedge dx_n = g_{x_i} \omega \in \W_1 \Omega^n_R
$$
vanishes as well.
\end{proof}

\begin{proof}[Proof of Proposition \ref{prop:omega_vanishes}.]
By \ref{para:annihilator}, the annihilator of $\omega$ is an ideal of $k[x_1, \dots, x_n]$ which contains $I$, and which is therefore nonzero.  Choose a nonzero element $g$ of $\Ann(\omega)$ of minimal total degree, and suppose that $g$ is not constant.  If $g$ is a $p$-th power, then Proposition \ref{prop:radical_and_partials} implies that $g^{1/p}$ is a nonconstant element of $\Ann(\omega)$ of smaller total degree, a contradiction.  Otherwise, $g$ has at least one nonzero partial derivative, so again Proposition \ref{prop:radical_and_partials} leads to a contradiction.  Thus $g$ must be a constant, so $\omega = 0$ in $\W_1 \Omega^n_R$.
\end{proof}

\section{The general case}

In this section, we will bootstrap up from the case of hypersurfaces to prove our main results in full generality.  We begin with the following generalization of Proposition \ref{prop:omega_vanishes}:

\begin{proposition} \label{prop:omega_vanishes_general}
Suppose $R$ is a Noetherian $k$-algebra, and $n \in \N$.
If $n > \sup_{\eta}(\trdeg(\kappa(\eta)/k))$ (where $\eta$ runs over the generic points of the irreducible components of $\Spec R$), then the natural map $\Omega^n_R \to \W_1 \Omega^n_R$ vanishes identically.
\end{proposition}
\begin{proof}
Assume without loss of generality that $R$ is reduced.  It suffices to prove that for every $n$-tuple of elements $x_1, \dots, x_n \in R$, the element $dx_1 \wedge \cdots \wedge dx_n \in \Omega^n_R$ maps to $0 \in \W_1 \Omega^n_R$.  Accordingly, fix $x_1, \dots, x_n$, and let $\pi \colon k[t_1, \dots, t_n] \to R$ be the $k$-algebra homomorphism determined by $t_i \mapsto x_i$ for each $i$.  We first argue that $\pi$ is not injective.  To see this, first note that $\Spec R$ has finitely many irreducible components by the Noetherian hypothesis, and by the transcendence degree assumption, each of the induced maps $k[t_1, \dots, t_n] \to \kappa(\eta_i)$ is non-injective.  Thus, by multiplying together nonzero polynomials which vanish in each of the various fraction fields, we can find a nonzero polynomial $g \in k[t_1, \dots, t_n]$ which vanishes identically on $R$; that is, $g \in \ker \pi$.

Then $\pi$ factors through $k[t_1, \dots, t_n]/(g)$, so we have a commutative diagram of commutative differential graded algebras
\[
\begin{tikzcd}
{\Omega^*_{k[t_1, \dots, t_n]}} \arrow[r] \arrow[d] & {\Omega^*_{k[t_1, \dots, t_n]/(g)}} \arrow[r] \arrow[d] & \Omega^*_R \arrow[d] \\
{\W_1 \Omega^*_{k[t_1, \dots, t_n]}} \arrow[r]      & {\W_1 \Omega^*_{k[t_1, \dots, t_n]/(g)}} \arrow[r]      & \W_1 \Omega^*_R.
\end{tikzcd}
\]
By Proposition \ref{prop:omega_vanishes}, the element $dt_1 \wedge \cdots \wedge dt_n \in \Omega^n_{k[t_1, \dots, t_n]}$ vanishes upon passage to $\W_1 \Omega^*_{k[t_1, \dots, t_n]/(g)}$.
Since the same element maps to $dx_1 \wedge \cdots \wedge dx_n \in \Omega^n_R$, it follows that this element vanishes in $\W_1 \Omega^*_R$, as claimed.
\end{proof}

We are now ready to prove our main results.

\begin{theorem} \label{thm:noetherian_dimensional_vanishing}
Suppose $R$ is a Noetherian $k$-algebra.  The complex $\W\Omega^*_R$ is concentrated in degrees $[0, \sup_{\eta}(\trdeg(\kappa(\eta)/k))]$, where $\eta$ runs over the generic points of $\Spec R$ as before.
\end{theorem}
\begin{proof}
Combine Proposition \ref{prop:omega_vanishes_general} with Lemma \ref{lem:vanishing_map_suffices}.
\end{proof}

\begin{corollary} \label{cor:dimensional_vanishing_finite-type}
Suppose $R$ is a finite-type $k$-algebra.  Then the complex $\W\Omega^*_R$ is concentrated in degrees $[0, \dim R]$.
\end{corollary}
\begin{proof}
If $R$ is a finite-type $k$-algebra, then $R$ is Noetherian, and the transcendence degree of each $\kappa(\eta)$ over $k$ is the dimension of the corresponding irreducible component of $\Spec R$.  Thus the upper bound $\sup_{\eta}(\trdeg(\kappa(\eta)/k))$ from Theorem \ref{thm:noetherian_dimensional_vanishing} is equal to $\dim R$.
\end{proof}

Since the saturated de Rham--Witt complex is computed on affines, the corresponding global statements follow immediately:
\begin{corollary} \label{cor:dimensional_vanishing_global}
\begin{enumerate}
\item For any locally Noetherian $k$-scheme $X$, the complex $\W\Omega^*_X$ is concentrated in degrees $[0, \sup_{\eta}(\trdeg(\kappa(\eta)/k))]$, where $\eta$ runs over the generic points of the irreducible components of $X$.
\item If $X/k$ is locally of finite type, then $\W\Omega^*_X$ is concentrated in degrees $[0, \dim X]$.
\end{enumerate}
\end{corollary}

As for hypercohomology, we have:

\begin{corollary} \label{cor:dimensional_vanishing_coh}
For any finite-type scheme $X/k$, the hypercohomology $\mathbb H^*(X, \W\Omega^*_X)$ is concentrated in degrees $[0, 2 \dim X]$.  If $X/k$ is moreover affine, then $\mathbb H^*(X, \W\Omega^*_X)$ is concentrated in degrees $[0, \dim X]$.  The same bounds hold for the hypercohomology of the quotient complex $\W_r \Omega^*_X$ for any $r > 0$.
\end{corollary}
\begin{proof}
The affine case follows from Corollary \ref{cor:dimensional_vanishing_finite-type}, since quasicoherent sheaves are acyclic on affines.
The general case follows from considering the spectral sequence
$$
E_1^{i,j} = H^j(X, \W\Omega^i_X) \implies \mathbb H^{i+j}(X, \W\Omega^*_X)
$$
and its finite-length analogues, where Corollary \ref{cor:dimensional_vanishing_global} and Grothendieck's dimensional vanishing theorem imply that $E_1^{i,j} = 0$ unless $0 \leq i, j \leq \dim X$.
\end{proof}

Finally, let us state a slight generalization of the above, which gives a good bound even for schemes which are of finite type over imperfect fields such as $k(x_1, \dots, x_n)$.  The proof is completely analogous, using Theorem \ref{thm:noetherian_dimensional_vanishing} instead of its corollaries.

\begin{corollary} \label{cor:dimensional_vanishing_coh_noetherian}
Suppose $k$ is a perfect field and $X/k$ is a Noetherian scheme.
Then the hypercohomology $\mathbb H^*(X, \W\Omega^*_X)$ is concentrated in degrees $[0, \dim X + \sup_{\eta}(\trdeg(\kappa(\eta)/k))]$.  If $X/k$ is moreover affine, then $\mathbb H^*(X, \W\Omega^*_X)$ is concentrated in degrees $[0, \sup_{\eta}(\trdeg(\kappa(\eta)/k))]$.  The same bounds hold for the hypercohomology of the quotient complex $\W_r \Omega^*_X$ for any $r > 0$.
\end{corollary}

\section{Future directions}

\begin{para}
There are several known $p$-adic cohomology theories which enjoy comparisons to crystalline cohomology in the smooth proper case:  crystalline cohomology itself, classical and saturated de Rham--Witt cohomology, rigid cohomology, and so on.  Among these, the only ones which obey dimensional vanishing theorems beyond the smooth case are saturated de Rham--Witt cohomology (Corollary \ref{cor:dimensional_vanishing_coh}) and rigid cohomology (\cite[Proposition 8.2.21]{le_stum}).
This can be seen as evidence for a positive answer to the following question of Illusie (\cite{illusie_talks}), which would be interesting to study in future work.
\end{para}

\begin{question}
If $X/k$ is proper and $K = W(k)[\frac{1}{p}]$, do we have an isomorphism
$$
R\Gamma(X, \W\Omega^*_X) \otimes_{W(k)} K \simeq R\Gamma_{\mathrm{rig}}(X/K)?
$$
That is, is the cohomology of $\W\Omega^*_X$ an integral model for rigid cohomology?
\end{question}

\printbibliography

\end{document}